\documentclass[reqno,centertags,12pt]{amsart}
\usepackage[letterpaper,margin=1.3in]{geometry}
\usepackage{amsmath,amsthm,amscd,amssymb}
\usepackage{enumerate}
\usepackage[nobysame,abbrev]{amsrefs} 
\usepackage{hyperref}

\newcommand{\bbC}{{\mathbb{C}}}
\newcommand{\bbD}{{\mathbb{D}}}

\newcommand{\bbR}{{\mathbb{R}}}

\newcommand{\fre}{{\mathfrak{e}}}

\newcommand{\al}{\alpha}
\newcommand{\be}{\beta}

\newcommand{\om}{\omega}
\newcommand{\Om}{\Omega}
\newcommand{\ze}{\zeta}

\newcommand{\calP}{{\mathcal P}}

\newcommand{\no}{\nonumber}
\newcommand{\lb}{\label}

\newcommand{\ol}{\overline}
\newcommand{\bi}{\bibitem}

\newcommand{\wti}{\widetilde}

\newcommand{\st}{:}
\newcommand{\sgn}{\mathrm{sgn}}
\newcommand{\ord}{\mathrm{ord}}

\newcommand{\cvh}{\text{\rm{cvh}}}

\newcommand{\beq}{\begin{equation}}
\newcommand{\eeq}{\end{equation}}
\newcommand{\ba}{\begin{align}}
\newcommand{\ea}{\end{align}}
\newcommand{\eps}{\varepsilon}

\newcommand{\norm}[1]{\lVert#1\rVert}

\newcommand{\bs}{\backslash}
\newcommand{\pd}{\partial}
\DeclareMathOperator{\ca}{C}

\newcounter{smalllist}

\newtheorem{theorem}{Theorem}[section]

\newtheorem{lemma}[theorem]{Lemma}
\newtheorem{corollary}[theorem]{Corollary}

\theoremstyle{definition}
\newtheorem*{remark}{Remark}

\allowdisplaybreaks
\numberwithin{equation}{section} 

\begin{document}

\title[Bounds for weighted Chebyshev and residual polynomials]{Bounds for weighted Chebyshev and residual polynomials on subsets of $\bbR$}

\author[J.~S.~Christiansen, B.~Simon and M.~Zinchenko]{Jacob S.~Christiansen$^{1,4}$, Barry Simon$^{2,5}$ \\and
Maxim~Zinchenko$^{3,6}$}

\thanks{$^1$ Centre for Mathematical Sciences, Lund University, Box 118, 22100 Lund, Sweden.
E-mail: jacob{\_}stordal.christiansen@math.lth.se}

\thanks{$^2$ Departments of Mathematics and Physics, Mathematics 253-37, California Institute of Technology, Pasadena, CA 91125.
E-mail: bsimon@caltech.edu}

\thanks{$^3$ Department of Mathematics and Statistics, University of New Mexico,
Albuquerque, NM 87131, USA; E-mail: maxim@math.unm.edu}

\thanks{$^4$ Research supported by VR grant 2023-04054 from the Swedish Research Council and in part by DFF research project 1026-00267B from the Independent Research Fund Denmark.}

\thanks{$^5$ Research supported in part by Israeli BSF Grant No.~2020027.}

\thanks{$^6$ Research supported in part by Simons Foundation grant MP--TSM--00002651.}

\dedicatory{Dedicated to the 80th birthday of Ed Saff.}

\date{\today}

\begin{abstract}
We give upper and lower bounds for weighted Chebyshev and residual polynomials on subsets of the real line. 
As an application, we prove a Szeg\H{o}-type theorem in the setting of Parreau--Widom sets.

\end{abstract}

\maketitle

\section{Introduction}

Let $\fre\subset\bbC$ be an infinite compact set and $w:\fre\to[0,\infty)$ a weight function which we assume to be upper semi-continuous and positive at infinitely many points of $\fre$. We denote the sup norm over $\fre$ by $\|\cdot\|_\fre$. The $n$th weighted Chebyshev polynomial, denoted $T_{n,w}$, is the unique polynomial that minimizes $\|wP_n\|_\fre$ among all monic polynomials $P_n$ of degree $n$. We set
\begin{align}\label{tn-cheb}
  t_n(\fre,w)=\|wT_{n,w}\|_\fre
\end{align}
and for sets $\fre$ of positive logarithmic capacity, we define the Widom factors by
\begin{align}\label{Wn-cheb}
  W_n(\fre,w) := \frac{t_n(\fre,w)}{\ca(\fre)^n}.
\end{align}

In the following, we will always assume that $\ca(\fre)>0$ and refer the reader to, e.g., \cite{ArmGar01, Hel09, Lan72, MF06, Ran95} for the basics of potential theory. 
In the unweighted case (i.e., when $w\equiv1$), Szeg\H{o} \cite{Sze1924} showed that the Widom factors satisfy a universal lower bound for subsets of the complex plane, namely
\begin{align}\label{SzLB}
  W_n(\fre,1) \ge 1, \quad n\ge1.
\end{align}
For subsets of the real line, Schiefermayr \cite{Sch08} proved the doubled lower bound
\begin{align}\label{SchLB}
  W_n(\fre,1) \ge 2, \quad n\ge1.
\end{align}
Regarding upper bounds, we showed in \cite{CSZ1} (see also \cite{CSYZ2} and \cite{And16, And17, AN19, ToVa15}) that for real subsets $\fre$ that are regular for potential theory, the following holds:
\begin{align}\label{UB}
  W_n(\fre,1) \le 2\exp\big[PW(\fre)], \quad n\ge1.
\end{align}
Here, $PW(\fre)$ is the Parreau--Widom constant 
given by
\begin{align}\label{PW-cheb}
  PW(\fre):=\sum_{\ell} g_\fre(\ze_\ell,\infty),
\end{align}
where $g_\fre(z,z_0)$ denotes the Green function of the domain $\ol\bbC\bs\fre$ with a logarithmic pole at $z_0$ and $\{\ze_\ell\}$ are the critical points of $g_\fre(z,\infty)$ in $\bbC\bs\fre$, that is, solutions of $\nabla g_\fre(z,\infty)=0$. When $\fre\subset\bbR$, the critical points $\{\ze_\ell\}$ are located in the gaps of $\fre$, with at most one per gap due to the concavity of $g_\fre(z,\infty)$. 
Following \cite{Has83}, a compact and regular set $\fre$ is referred to as a Parreau–Widom set if $PW(\fre)<\infty$.

Extensions of the above bounds to the weighted setting rely on the so-called Szeg\H{o} factor $S(\fre,w)$ of the weight $w$. Let $d\rho_\fre$ be the equilibrium measure for the set $\fre$. A weight $w$ is said to be of Szeg\H{o} class (for $\fre$) if
\begin{align} \lb{S-cheb}
  S(\fre,w) := \exp\left[\int\log w(z)\,d\rho_\fre(z)\right] > 0.
\end{align}
We note that, trivially, $S(\fre,1)=1$.

In the weighted setting, the Widom factors satisfy an analog of Szeg\H{o}'s lower bound \eqref{SzLB} for subsets of the complex plane. By \cite[Chap.~I, Thm.~3.6]{ST97} and \cite[Thm.~13]{NSZ21}, we have
\begin{align} \label{wSzLB-cheb}
  W_n(\fre,w) \ge S(\fre,w), \quad n\ge1.
\end{align}
However, unlike the unweighted case, the lower bound \eqref{wSzLB-cheb} is already optimal for subsets of the real line, see \cite[Thm.~13]{NSZ21}. In particular, the doubling of the lower bound as in \eqref{SchLB} vs \eqref{SzLB} does not generally hold in the weighted setting.
Nevertheless, it was observed in \cite{SZ21,Alp22,AZ24} that the doubled version of \eqref{wSzLB-cheb} does hold for some special weights. In particular, for $\fre\subset\bbR$ and weights of the form $w_0(x)=1/|P_m(x)|$, where $P_m$ is a polynomial of degree $m$, it is shown in \cite{AZ24} that
\begin{align}\label{wLB-cheb}
  W_n(\fre,w_0) \ge 2S(\fre,w_0), \quad n>m.
\end{align}

In this paper, we derive a weighted analog of the upper bound \eqref{UB} for weights of the form $w_0(x)=1/|P_m(x)|$, where $P_m$ is a real polynomial of degree $m$. In particular, we show that for regular Parreau--Widom sets $\fre\subset\bbR$, the following holds:
\begin{align}\label{wUB-cheb}
  W_n(\fre,w_0) \le 2S(\fre,w_0)\exp\big[PW(\fre)\big], \quad n>m.
\end{align}
For sets $\fre\subset[-1,1]$, the upper bound \eqref{wUB-cheb} with $m=0$ was obtained earlier for the weight $w_0(x)=\sqrt{1-x^2}$ in \cite{SZ21} and for the weights $w_0(x)=\sqrt{1\pm x}$ in \cite{Alp22}.

From the non-asymptotic bounds \eqref{wLB-cheb} and \eqref{wUB-cheb}, we proceed to derive asymptotic bounds for more general weights. Specifically, for weights $w$ that are continuous at a.e.\ point of $\fre$ and satisfy $w\ge|P|$ on $\fre$ for some polynomial $P\not\equiv 0$, we show the asymptotic lower bound
\begin{align}\label{wALB-cheb}
  \liminf_{n\to\infty} W_n(\fre,w) \ge 2S(\fre,w).
\end{align}
For finite gap sets $\fre\subset\bbR$, the asymptotic bound \eqref{wALB-cheb} is derived in \cite{AZ24} for arbitrary continuous weights $w$ with at most finitely many zeros, and also for some special weights with infinitely many zeros. Whether or not \eqref{wALB-cheb} holds for all continuous weights $w$ is an open problem.

For regular Parreau--Widom sets $\fre\subset\bbR$ and arbitrary upper semi-continuous weights $w$, we show the asymptotic upper bound
\begin{align}\label{wAUB-cheb}
  \limsup_{n\to\infty} W_n(\fre,w) \le 2S(\fre,w)\exp\big[PW(\fre)\big].
\end{align}
In \cite[Sect.~11]{Wid69}, Widom established this asymptotic bound 
for complex sets $\fre$ consisting of finitely many disjoint $C^{2+}$ Jordan curves and/or arcs.

By combining the bounds in \eqref{wSzLB-cheb} and \eqref{wAUB-cheb}, we get a Szeg\H{o}-type theorem
(see also \cite{JSC} for a similar result on orthogonal polynomials):
\begin{theorem}
  Let $\fre\subset\bbR$ be a compact and regular Parreau--Widom set, and let $w$ be an upper semi-continuous weight on $\fre$. Then
\begin{align}
  (a) \; \inf_{n}W_n(\fre,w) > 0 \quad\Longleftrightarrow\quad
  (b) \; \int\log w(z)\,d\rho_\fre(z)>-\infty.
\end{align}
Moreover, if either $(a)$ or $(b)$ holds, then also $\sup_{n}W_n(\fre,w)<\infty$.
\end{theorem}

Weighted Chebyshev polynomials are $L^\infty(\fre,w)$ extremal polynomials normalized at infinity. It is worth discussing them in the context of a larger family of $L^\infty(\fre,w)$ extremal polynomials normalized at a general point $x_*\in\ol\bbC\bs\fre$.
When the normalization point $x_*$ is finite, such extremal polynomials are called residual polynomials while their duals (polynomials of unit norm that maximize the value at $x_*$) are called polynomials of extremal growth.
Such polynomials arise naturally in the analysis of Krylov subspace iterations, see, for example, \cite{DTT98,Fis96,Kui06}. For prior studies on residual polynomials/polynomials of extremal growth and their applications, we refer to \cite{FR86,Fre88,FF90,FF91,Fis92,Yud99,Peh09,Sch11, Eic17,EY18,EY21,CSZ5,ELY24,BLO21,CLW24}.

The $n$th weighted residual polynomial, denoted $T_{n,w,x_*}$, is the unique polynomial that minimizes $\|wP_n\|_\fre$ among all polynomials $P_n$ of degree at most $n$ normalized by $P_n(x_*)=1$. We denote the minimal norms by
\begin{align}\label{tn-res}
  t_n(\fre,w,x_*)=\|wT_{n,w,x_*}\|_\fre
\end{align}
and define the corresponding Widom factors by
\begin{align}\label{Wn-res}
  W_n(\fre,w,x_*) := t_n(\fre,w,x_*)e^{ng_\fre(x_*,\infty)}.
\end{align}
The Widom factors introduced in \eqref{Wn-cheb} and \eqref{Wn-res} are compatible in the sense that
\begin{align}\label{Wn-cont}
  \lim_{x_*\to\infty} W_n(\fre, w, x_*) = W_n(\fre,w).
\end{align}
Because of this, we set $W_n(\fre,w,\infty)=W_n(\fre,w)$ and treat the Chebyshev and residual polynomials in a unified way.

In the following, we let $\Om$ be the unbounded component of $\ol\bbC\bs\fre$ and assume that $x_*\in\Om$.
The Szeg\H{o} factor suitable for the $L^\infty(\fre,w)$ extremal polynomials normalized at $x_*$ is given by
\begin{align} \lb{S-res}
  S(\fre,w,x_*) := \exp\left[\int\log w(z)\,d\om_\fre(z,x_*)\right],
\end{align}
where $d\om_\fre(\cdot,x_*)$ is the harmonic measure for the domain $\Om$. We recall from \cite[Thm.~4.3.14]{Ran95} that $d\rho_\fre=d\om_\fre(\cdot,\infty)$. Furthermore, by \cite[Cor.~4.3.5]{Ran95}, all harmonic measures for $\Om$ are comparable, meaning that there exist constants $c_1(x_*)$ and $c_2(x_*)$ such that
\begin{align}
   c_1(x_*)\rho_\fre\le\om_\fre(\cdot,x_*)\le c_2(x_*)\rho_\fre.
\end{align}
It follows that a weight $w$ is of Szeg\H{o} class for $\fre$ if and only if $S(\fre,w,x_*)>0$ for one and hence all $x_*\in\Om$.

For complex sets $\fre$, the Widom factors $W_n(\fre,w,x_*)$ satisfy a weighted analog of Szeg\H{o}'s lower bound, namely
\begin{align}\label{wSzLB-res}
  W_n(\fre,w,x_*) \ge S(\fre,w,x_*), \quad n\ge1.
\end{align}
This bound is an immediate consequence of the weighted analog of the Bernstein--Walsh inequality \cite[Thm.~12]{NSZ21} stating that
\begin{align}
|P_n(z)| \le \|wP_n\|_\fre e^{ng_\fre(z,\infty)}/S(\fre,w,z), \quad z\in\Om,
\end{align}
whenever $\fre\subset\bbC$ is a compact set of positive capacity, $w$ a Szeg\H{o} class weight on $\fre$, and $P_n$ an arbitrary polynomials of degree $n$.

As is the case for Chebyshev polynomials, when $\fre\subset\bbR$ and $x_*\in\bbR\bs\fre$, there are better bounds for the residual polynomials. In the setting of unweighted residual polynomials, an improved lower bound and an upper bound analogous to \eqref{SchLB} and \eqref{UB} were derived in \cite{Sch11,CSZ5} for a modified version of the Widom factors. In the weighted setting, that modified version of the Widom factors does not appear to be suitable neither for a lower nor an upper bound. Instead, we will show that the analog of the upper bound \eqref{wUB-cheb} holds for the Widom factors $W_n(\fre, w, x_*)$ defined in \eqref{Wn-res}. Moreover, our analog of the lower bound \eqref{wLB-cheb} has an additional factor, as will be explained in Section~\ref{s2}.

Following ideas of \cite{ELY24}, we perform our analysis in the unified setting that covers both residual and Chebyshev polynomials. We derive non-asymptotic upper and lower bounds in Section~\ref{s2}, while Section~\ref{s3} presents the corresponding asymptotic bounds along with a Szeg\H{o}-type theorem.

Ed Saff was a pioneer in the modern theory of general Chebyshev polynomials.  We hope he enjoys this birthday bouquet that we are pleased to dedicate to him.

\section{Non-asymptotic Bounds}\label{s2}

Throughout this section, we assume that $\fre\subset\bbR$ is a compact set and $x_*$ a point in $\ol\bbR\bs\fre$, where $\ol\bbR:=\bbR\cup\{\infty\}$. In this setting, the $L^\infty(\fre,w)$ extremal polynomials $T_{n,w,x_*}$ are real (i.e., have real coefficients) and characterized by the following version of the alternation theorem.

For convenience of notation, we denote by $\calP_n(x_*)$ the set of polynomials $P$ of degree at most $n$ normalized at $x_*$ so that $P(x_*)=1$ if $x_*$ is finite and so that $P$ is monic of degree $n$ if $x_*=\infty$. We also denote by $(b,a)$, with $a<b$, the set $\ol\bbR\bs[a,b]$.

\begin{theorem}[Alternation Theorem]
  Let $\fre\subset\bbR$ be a compact set, $x_*$ a point in $\ol\bbR\bs\fre$, $w$ an upper semi-continuous weight function on $\fre$, and $P\in\calP_n(x_*)$ a real polynomial. Then $P=T_{n,w,x_*}$ if and only if there exist $n+1$ alternation points $x_1<x_2<\dots<x_{n+1}$ on $\fre$ such that
\begin{align}\label{alt-cond}
  w(x_j)P(x_j) = (-1)^{k_*-j}\sgn(x_*-x_j)\|wP\|_\fre, \quad j=1,\dots,n+1,
\end{align}
where $k_*\in\{1,\dots,n+1\}$ is the index for which $x_*\in(x_{k_*},x_{k_*+1})$, $x_{n+2}=x_1$, and $\sgn(\infty-x_j)=1$ for every $j$.
\end{theorem}
\begin{proof}
Suppose that $P\in\calP_n(x_*)$ is a real polynomial such that \eqref{alt-cond} holds.  If $P$ is not a norm minimizer, then $\|wT_{n,w,x_*}\|_\fre<\|wP\|_\fre$.  Consider the polynomial $Q(x)=P(x)-T_{n,w,x_*}(x)$ if $x_*=\infty$ and, otherwise,
\[
   Q(x)=\frac{P(x)-T_{n,w,x_*}(x)}{x-x_*}.
\]
It has degree at most $n-1$ and alternating signs at $x_1,\dots,x_{n+1}$, hence a zero in each of the intervals $(x_j,x_{j+1})$, $j=1,\dots,n$.  It follows that $Q$ is identically zero, a contradiction.  Thus $P$ is a norm minimizer.

Conversely, suppose $P\equiv T_{n,w,x_*}$ and $\sgn(x_*-x)P(x)$ has at most $n-1$ sign changes on the set of extreme points of $P$, that is,
\[
   \bigl\{x\in\fre:w(x)|P(x)|=\|wP\|_\fre\bigr\}.
\]
Then, by putting zeros in the right places, there exists a polynomial $Q_0$ of degree at most $n-1$ such that on the set of extreme points, $\sgn(Q_0(x))=\sgn(P(x))$ if $x_*=\infty$ and otherwise $\sgn(Q_0(x))=\sgn(P(x)/(x_*-x))$.  Let $Q(x):=Q_0(x)$ if $x_*=\infty$ and otherwise $Q(x):=(x_*-x)Q_0(x)$. Then $(P-\eps Q)\in\calP_n(x_*)$ and $\norm{P-\varepsilon Q}_\fre<\norm{P}_\fre$ for sufficiently small $\varepsilon>0$, contradicting the fact that $P$ is a norm minimizer.  Therefore, an alternating set exists.
\end{proof}

We call the connected components of $\ol\bbR\bs\fre$ the gaps of $\fre$ and refer to $\ol\bbR\bs\cvh(\fre)$ as the unbounded gap of $\fre$. In addition, we say that $T_{n,w,x_*}$ has a zero at $\infty$ if $\deg(T_{n,w,x_*})<n$.

Since the weight $w$ is necessarily positive at the alternation points, it follows from the intermediate value theorem that $T_{n,w,x_*}$ has a zero in $(x_j,x_{j+1})$ for each $j=1,\dots,k_*-1,k_*+1,\dots,n+1$.
This implies that $T_{n,w,x_*}$ has degree at least $n-1$, that all the zeros of $T_{n,w,x_*}$ are real and simple, and that $T_{n,w,x_*}$ has no zeros in $(x_{k_*},x_{k_*+1})$. In particular, if $x_*$ lies in the unbounded gap of $\fre$, then $\deg(T_{n,w,x_*})=n$. In addition, if $\deg(T_{n,w,x_*})<n$, then $T_{n,w,x_*}$ has no zeros in $(x_{n+1},x_1)$ and hence it has no zeros in the unbounded gap of $\fre$.

\begin{corollary}
The extremal polynomials $T_{n,w,x_*}$, and consequently the Widom factors $W_n(\fre,w,x_*)$, depend continuously on $x_*$ within each gap of $\fre$. In particular, \eqref{Wn-cont} holds.
\end{corollary}
\begin{proof}
It follows from the alternation theorem that if $x$ and $x_*$ lie in the same gap of $\fre$, then $T_{n,w,x}(z)=T_{n,w,x_*}(z)/T_{n,w,x_*}(x)$ and hence also $W_n(\fre,w,x)=W_n(\fre,w,x_*)/T_{n,w,x_*}(x)$.
\end{proof}

Next, we consider weights of the form
\begin{align}\label{wRat-res}
  w_0(x)=\frac1{|P_m(x)|},
\end{align}
where $P_m(x)=c\prod_{j=1}^m(x-c_j)$ is a real polynomial of degree $m$ with zeros outside of $\fre$. We note that $P_m$ is assumed to have real coefficients but not necessarily real zeros. Given such a weight $w_0$, we define
\begin{align}
  n_0:=2(m+1)-\deg(T_{m+1,w,x_*}) \in \{m+1,m+2\}
\end{align}
and introduce, for $n\ge n_0$, the following rational function
\begin{align}\label{R-res}
  R_n(z) = \pm\frac{T_{n,w,x_*}(z)}{P_m(z)}.
\end{align}
We choose the sign such that $R_n(x_*)>0$ if $x_*$ is not a pole of $R_n$ and otherwise such that the leading coefficient of the principal part of the Laurent series of $R_n$ at $x_*$ is positive, that is,
\begin{align}
\begin{split}
   \lim_{z\to x_*}R_n(z)(z-x_*)^{\ord(R_n,x_*)}>0 
   \;\text{ if }\; x_*\neq\infty,\\
   \lim_{z\to x_*}R_n(z)z^{-\ord(R_n,x_*)}>0 
   \;\text{ if }\; x_*=\infty,
\end{split}
\end{align}
where $\ord(R_n,x_*)$ denotes the order of the pole of $R_n$ at $x_*$.

The above choice of $n_0$ guarantees that $R_n$ has a pole at infinity for all $n\ge n_0$. By construction, we also have that
\begin{align}
   \|R_n\|_\fre=\|wT_{n,w,x_*}\|_\fre=t_n(\fre,w,x_*).
\end{align}
It might happen that the numerator and denominator in \eqref{R-res} have common zeros. Without loss of generality, we assume that $\{c_j\}_{j=1}^{r_n}$, $r_n\le m$, are the poles of $R_n$ in $\bbC$, repeated according to their multiplicity, and $\{c_j\}_{j=r_n+1}^m$ are the zeros of $T_{n,w,x_*}$ that coincide with the zeros of $P_m$. For future use, we also define
\begin{align}
  d_n := \deg(R_n)
\end{align}
and note that it follows from \eqref{R-res} and the definitions of $r_n$ and $d_n$ above that
\begin{align}\label{dn-rn}
  d_n-r_n = \deg(T_{n,w,x_*})-m \le n-m,
\end{align}
since the quantity on the left-hand side equals the order of the pole of $R_n$ at $\infty$.

In what follows, an important role will be played by the compact set
\begin{align}\label{En-res}
  \fre_n := \big\{z\in\bbC \st R_n(z) \in [-t_n(\fre,w,x_*),t_n(\fre,w,x_*)] \big\}.
\end{align}
By construction, we clearly have that $\fre\subset\fre_n$. Using the alternation theorem, we can derive several important properties of the set $\fre_n$.

\begin{theorem}\label{En-thm}
Let $\fre\subset\bbR$ be a compact set and suppose $n\ge n_0$. Then there are points $\al_1<\be_1\le\al_2<\dots<\be_j\le\al_{j+1}<\dots<\be_{d_n}$ in $\bbR$ such that
\begin{align}\label{EnBands-res}
  \fre_n=\bigcup_{j=1}^{d_n}[\al_j,\be_j],
\end{align}
and $R_n$ maps each of the intervals $[\al_j,\be_j]$ bijectively onto $[-t_n(\fre,w,x_*),t_n(\fre,w,x_*)]$. In addition,
\begin{itemize}
\item[$(i)$]
each gap of $\fre$ may intersect at most one band $[\al_j,\be_j]$ of $\fre_n$,
\item[$(ii)$]
the gaps of $\fre$ containing $x_*$ and $\{c_j\}_{j=r_n+1}^m$ do not intersect $\fre_n$, and
\item[$(iii)$]
if $\deg(T_{n,w,x_*})<n$, then the unbounded gap of $\fre$ does not intersect $\fre_n$.
\end{itemize}
\end{theorem}
\begin{proof}
Let $x_1<\dots<x_{n+1}$ be $x_*$-alternation points on $\fre$ for the residual polynomial $T_{n,w,x_*}$.
Since every zero of $R_n$ is a zero of $T_{n,w}$, all $d_n$ zeros of $R_n$ are simple, real, and situated in $d_n$ distinct sets $I_j=(x_{k_j},x_{k_j+1})$, $j=1,\dots,d_n$. In each closed set $\ol{I}_j$, let $z_j$ be the zero of $R_n$ and $\al_j,\be_j$ the points closest to $z_j$ satisfying
\[
   \al_j<z_j<\be_j \; \mbox{ and } \; |R_n(\al_j)|=|R_n(\be_j)|=t_n(\fre,w,x_*).
\]
Then $R_n$ has no poles but only a single simple zero in $(\al_j,\be_j)$. This implies that $R_n(\al_j)$ and $R_n(\be_j)$ are of opposite sign and hence, by the intermediate value theorem, $R_n$ maps each $(\al_j,\be_j)$ onto $(-t_n(\fre,w,x_*),t_n(\fre,w,x_*))$. Since $R_n$ takes every value $d_n$ times, $R_n$ is one-to-one on each $(\al_j,\be_j)$, $j=1,\dots,d_n$. Thus, $R_n$ maps each interval $[\al_j,\be_j]$ bijectively onto $[-t_n(\fre,w,x_*),t_n(\fre,w,x_*)]$ and \eqref{EnBands-res} holds. Additionally, we can rearrange the bands in increasing order by relabeling them if necessary.

By construction, each $[\be_j,\al_{j+1}]$ as well as $[\be_{d_n},\al_1]$ contains at least one of the $x_*$-alternation points. Hence, a gap of $\fre$ can not  intersect two bands of $\fre_n$. Finally, due to cancelation, $R_n$ does not have zeros at $\{c_j\}_{j=r_n+1}^m$ so the intervals $(x_{k_j},x_{k_j+1})$ containing $\{c_j\}_{j=r_n+1}^m$, and hence the gaps of $\fre$ containing $\{c_j\}_{j=r_n+1}^m$, do not intersect $\fre_n$. Since $T_{n,w,x_*}$ and hence also $R_n$ do not have zeros in the interval $(x_{k_*},x_{k_*+1})\ni x_*$, the gap of $\fre$ containing $x_*$ does not intersect $\fre_n$ as well.
\end{proof}

Recall that the Green function of the domain $\ol\bbC\bs\fre_n$ with a logarithmic pole at $z_0$ is denoted $g_{\fre_n}(z,z_0)$. We will work with its complexified exponential variant given by
\begin{align}
  B_{\fre_n}(z,z_0)=\exp[-g_{\fre_n}(z,z_0)-\wti g_{\fre_n}(z,z_0)],
\end{align}
where $\wti g_{\fre_n}(z,z_0)$ is a harmonic conjugate of $g_{\fre_n}(z,z_0)$. In general, $B_{\fre_n}(\cdot,z_0)$ is a multiple-valued analytic character-automorphic function on $\ol\bbC\bs\fre_n$. However, on $\ol\bbC\bs(x_{k_*},x_{k_*+1})$, it is single-valued, and we normalize its phase by imposing the condition  $B_{\fre_n}(x_*,z_0)>0$ for $z_0\neq x_*$, and, in the case $z_0 = x_*$, by requiring that
 $\lim_{z\to x^*}B_{\fre_n}(z,x^*)/|z-x^*|>0$.

The key ingredient for our results is the Blaschke-type function defined by
\begin{align}\label{Bn-res}
  B_n(z) := B_{\fre_n}(z,\infty)^{d_n-r_n}\prod_{j=1}^{r_n}B_{\fre_n}(z,c_j)
\end{align}
and its connection to the rational function $R_n$, as outlined in the following theorem.

\begin{theorem}
On the domain $\ol\bbC\bs\fre_n$, the function $B_n$ is single-valued and
\begin{align}\label{RnBn-res}
  R_n(z) = \frac{t_n(\fre,w,x_*)}2\bigg(B_n(z)+\frac1{B_n(z)}\bigg).
\end{align}
\end{theorem}
\begin{proof}
Denote by $J(z)=\frac12(z+1/z)$ the Joukowsky map, and recall that $J$ is a conformal mapping from $\bbD$ onto $\ol\bbC\bs[-1,1]$. It follows directly from the definition of the set $\fre_n$ in \eqref{En-res} that the function $\Psi_n(z):=J^{-1}(R_n(z)/t_n(\fre,w,x_*))$ maps $\ol\bbC\bs\fre_n$ onto $\bbD$, and $\lim_{z\to x}|\Psi(z)|=1$ for every $x\in\fre_n$. Moreover, it easy to see that $\{c_j\}_{j=1}^{r_n}$ account for all the zeros of $\Psi_n$. It thus follows from the maximum principle that $|\Psi_n|=|B_n|$, and hence $\Psi_n=cB_n$ on $\bbC\bs\fre_n$ for some unimodular constant $c$. Our choice of normalization of $B_{\fre_n}$ and $R_n$ at $x_*$ implies that $c=1$ and hence $R_n(z)=t_n(\fre,w,x_*)J(B_n(z))$ which is precisely \eqref{RnBn-res}.
\end{proof}

An important consequence of the above relation is the following constraint for the harmonic measures $\om_{\fre_n}(\cdot,z_0)$ of the bands of $\fre_n$ and the gaps of $\fre$.
\begin{corollary}
Let $I_\ell=[\al_\ell,\be_\ell]$ be a band of $\fre_n$ as in Theorem~\ref{En-thm}. Then
\begin{align}\label{om-id-res}
  (d_n-r_n)\om_{\fre_n}(I_\ell,\infty)+\sum_{j=1}^{r_n} \om_{\fre_n}(I_\ell,c_j) = 1.
\end{align}
In particular, since every gap $K$ of $\fre$ intersects at most one band $I_\ell$ of $\fre_n$, we have
\begin{align}\label{om-est-res}
  (d_n-r_n)\om_{\fre_n}(K,\infty)+\sum_{j=1}^{r_n} \om_{\fre_n}(K,c_j) \le 1.
\end{align}
\end{corollary}
\begin{proof}
By Theorem~\ref{En-thm}, $R_n(x)/t_n(\fre,w,x_*)$ maps $I_\ell$ monotonically onto $[-1,1]$. Since $J(z)$ maps the upper (resp. lower) half of the unit circle bijectively onto $[-1,1]$, it follows from \eqref{RnBn-res} that
\begin{align}
  \bigl|\arg B_n(\be_\ell)-\arg B_n(\al_\ell)\bigr|=\pi.
\end{align}
By the Cauchy--Riemann equations, we have
\begin{align}
  \arg B_n(\be_\ell)-\arg B_n(\al_\ell) = \int_a^b\frac{\pd g_n(x)}{\pd n}dx,
\end{align}
where
\begin{align}
  g_n(z) = -\log|B_n(z)| = (d_n-r_n)g_{\fre_n}(z,\infty)+\sum_{j=1}^{r_n}g_{\fre_n}(z,c_j).
\end{align}
Finally, since $\om_{\fre_n}(x,c) = \frac1\pi\frac{\pd g_{\fre_n}(x,c)}{\pd n}$, \eqref{om-id-res} follows from these last three identities.
\end{proof}
\begin{remark}
The above results are essentially contained in \cite[Section~2]{ELY24}, where they are stated for Chebyshev rational functions with real poles.
\end{remark}

Next, we introduce the Parreau--Widom constant of a compact set $\fre\subset\bbR$ relative to the point $x_*\in\ol\bbR\bs\fre$. This constant is defined as
\begin{align}\label{PW-RES}
  PW(\fre,x_*):=\sum_{\ell} g_\fre(\ze_\ell,x_*),
\end{align}
where $\{\ze_\ell\}$ are the critical points of $g_\fre(x,x_*)$ located in the gaps of $\fre$, that is, the points at which $\pd_x g_\fre(\ze_\ell,x_*)=0$. Recall that the derivative at infinity given by
\[
   \pd_x g_\fre(\infty,x_*)=\lim_{x\to\infty}x \bigl[g_\fre(x,x_*)-g_\fre(\infty,x_*)\bigr].
\]
As before, we say that $\fre$ is a Parreau--Widom set if $PW(\fre,x_*)<\infty$. It is known \cite[Chap.V]{Has83} that the Parreau--Widom condition is independent of the choice of $x_*$.

We note the following useful formula for the Szeg\H{o} factor introduced in \eqref{S-res}.
\begin{lemma}\label{Sw-pol-lem}
Let $\fre\subset\bbR$ be a compact set of positive capacity, $x_*$ a point in $\ol\bbR\bs\fre$, and $w_0$ a weight of the form \eqref{wRat-res}. If $x_*$ is finite and $P_m(x_*)\neq0$, then
\begin{align}\label{Sw-pol}
  S(\fre,w_0,x_*) = \frac1{|P_m(x_*)|}\exp\bigg[mg_\fre(x_*,\infty)-\sum_{j=1}^m g_\fre(x_*,c_j)\bigg].
\end{align}
If, on the other hand, $x_*=\infty$ or $P_m(x_*)=0$, then the Szeg\H{o} factor can be obtained from the above formula by taking the limit of $S(\fre,w_0,x)$ as $x\to x_*$. 
\end{lemma}
\begin{proof}
Since $\{c_j\}_{j=1}^m$ are the zeros of the polynomial $P_m(z)$, the function
\begin{align}
U(z):=mg_\fre(z,\infty)-\sum_{j=1}^m g_\fre(z,c_j)-\log|P_m(z)|
\end{align}
has a removable singularity at every $c_j$ and also at infinity. Hence, $U$ extends to a harmonic function on $\ol{\bbC}\bs\fre$. The boundary values of $U(z)$ are equal to $-\log|P_m(z)|=\log w_0(z)$ q.e. on $\fre$, and this implies that
\begin{align}
U(z)=\int \log w_0(\ze)d\om_\fre(\ze,z).
\end{align}
By \eqref{S-res}, the above integral equals $\log S(\fre,w_0,z)$ and taking $z\to x_*$ therefore yields \eqref{Sw-pol}.
\end{proof}

We are now in position to derive non-asymptotic lower and upper bounds for the Widom factors. This will be done in the following two theorems.

\begin{theorem}\label{wLB-thm}
Let $\fre\subset\bbR$ be a compact set of positive capacity, $x_*$ a point in $\ol\bbR\bs\fre$, and $w_0$ a weight of the form \eqref{wRat-res}. Then, for $n\ge n_0$, we have
\begin{align}\label{wLB-res}
  W_n(\fre,w_0,x_*) \ge \frac{2S(\fre,w_0,x_*)}{1+\exp[-2(n-m)g_\fre(x_*,\infty)]}.
\end{align}
\end{theorem}
\begin{proof}
Since both sides of \eqref{wLB-res} are continuous with respect to $x_*$, it is sufficient to prove the inequality for finite values of $x_*$ where $P_m(x_*)\neq0$.

Suppose $x_*$ is finite and satisfies $P_m(x_*)\neq0$. Evaluating \eqref{RnBn-res} at $z=x_*$ then gives
\begin{align}
  \frac{1}{|P_m(x_*)|} = \frac{t_n(\fre,w,x_*)}{2}\bigg(B_n(x_*)+\frac{1}{B_n(x_*)}\bigg).
\end{align}
Multiplying by $e^{ng_\fre(x_*,\infty)}/S(\fre,w,x_*)$ and using \eqref{Wn-res} and Lemma \ref{Sw-pol-lem}
gives us that
\begin{align}\label{Wn-id1}
  \frac{W_n(\fre,w,x_*)}{2S(\fre,w_0,x_*)} &= \frac{e^{ng_\fre(x_*,\infty)}}{|P_m(x_*)|S(\fre,w_0,x_*)} \bigg(B_n(x_*)+\frac{1}{B_n(x_*)}\bigg)^{-1} \no\\
  &= \frac{e^{(n-m)g_\fre(x_*,\infty)+\sum_{j=1}^m g_\fre(x_*,c_j)}}{B_n(x_*)+\frac{1}{B_n(x_*)}}.
\end{align}
By \eqref{Bn-res},
\begin{align}
  B_n(x_*)+\frac{1}{B_n(x_*)} = 2\cosh\Big((d_n-r_n)g_{\fre_n}(x_*,\infty) + \sum_{j=1}^{r_n}g_{\fre_n}(x_*,c_j)\Big).
\end{align}
Since $\fre\subset\fre_n$, we have $g_\fre(z,x_*)\ge g_{\fre_n}(z,x_*)$. This, together with \eqref{dn-rn}, $r_n\le m$, and the fact that $\cosh$ is increasing on $(0,\infty)$, implies that
\begin{align}
  B_n(x_*)+\frac{1}{B_n(x_*)} &\le 2\cosh\Big((n-m)g_{\fre}(x_*,\infty) + \sum_{j=1}^{m}g_{\fre}(x_*,c_j)\Big) \no \\
  &\le \big(1+e^{-2(n-m)g_{\fre}(x_*,\infty)}\big) e^{(n-m)g_{\fre}(x_*,\infty)+\sum_{j=1}^{m}g_{\fre}(x_*,c_j)}.
\end{align}
Substituting this estimate into \eqref{Wn-id1} then yields the lower bound \eqref{wLB-res}.
\end{proof}

\begin{theorem}\label{wUB-thm}
Let $\fre\subset\bbR$ be a regular and compact Parreau--Widom set, $x_*$ a point in $\ol\bbR\bs\fre$, and $w_0$ a weight of the form \eqref{wRat-res}. Then, for all $n\ge n_0$, we have
\begin{align}\label{wUB-res}
  W_n(\fre,w_0,x_*) < 2S(\fre,w_0,x_*)\exp[PW(\fre,x_*)].
\end{align}
\end{theorem}
\begin{proof}
Assume first that $x_*$ is finite and $P_m(x_*)\neq0$. Evaluating \eqref{RnBn-res} at $z=x_*$ and multiplying by $e^{ng_\fre(x_*,\infty)}B_n(x_*)/S(\fre,w_0,x_*)$ gives us
\begin{align}
  \frac{B_n(x_*)}{|P_m(x_*)|S(\fre,w_0,x_*)} = \frac{t_n(\fre,w_0,x_*)e^{ng_\fre(x_*,\infty)}}{2S(\fre,w_0,x_*)} \Big(1+B_n(x_*)^2\Big).
\end{align}
Then using \eqref{Wn-res}, \eqref{Bn-res}, \eqref{Sw-pol}, and the fact that $B_n(x_*)>0$, we obtain
\begin{align}\label{Wn-gn-ineq}
  \frac{W_n(\fre,w_0,x_*)}{2 S(\fre,w_0,x_*)} <
  \exp\bigg[&(n-m)g_\fre(\infty,x_*) - (d_n-r_n)g_{\fre_n}(\infty,x_*) \no\\
  &+ \sum_{j=1}^m g_\fre(c_j,x_*) - \sum_{j=1}^{r_n}g_{\fre_n}(c_j,x_*)\bigg].
\end{align}

Recall that $d_n-r_n=\deg(T_{n,w_0,x_*})-m$ is equal to $n-m$ if $x_*$ lies in the unbounded gap of $\fre$. Moreover, if $x_*$ and $c_j$ are in the same gap of $\fre$, then $T_{n,w_0,x_*}(c_j)\neq0$ so that $j\le r_n$. These observations imply that the right-hand side of \eqref{Wn-gn-ineq} is continuous with respect to $x_*$ on the gaps of $\fre$, in particular, at infinity and at every zero of $P_m$. Since the left-hand side of \eqref{Wn-gn-ineq} is also continuous with respect to $x_*$, by taking a limit, we can remove the restriction of $x_*$ being finite and $P_m$ being nonzero at this point.

Now, let $h(z)=g_\fre(z,x_*)-g_{\fre_n}(z,x_*)$ and note that 
\begin{align}\label{long-exp}
  &(n-m)g_\fre(\infty,x_*) - (d_n-r_n)g_{\fre_n}(\infty,x_*)
  + \sum_{j=1}^m g_\fre(c_j,x_*) - \sum_{j=1}^{r_n}g_{\fre_n}(c_j,x_*) \\
  &\quad = (d_n-r_n) h(\infty) + \sum_{j=1}^{r_n} h(c_j) +\bigl(n-\deg(T_{n,w_0,x_*})\bigr)g_\fre(\infty,x_*) + \sum_{j=r_n+1}^m g_\fre(c_j,x_*). \no
\end{align}
It suffices to show that this expression is bounded above by $PW(\fre,x_*)$.
We start by noting that $h(z)$ is harmonic on $\ol\bbC\bs\fre_n$ since the singularities of the Green functions at $x_*$ cancel. Therefore,
\begin{align}
   h(z)=\int h(x)d\om_{\fre_n}(x,z).
\end{align}
Let $\{K_\ell\}_{\ell\in I_1}$ be the gaps of $\fre$ that intersect $\fre_n$, so that $\fre_n\bs\fre\subset\cup_{\ell\in I_1} K_\ell$.
Since $\fre$ is regular, we have that $h(x)=0$ on $\fre$ and $h(x)=g_\fre(x,\infty)$ on $\fre_n\bs\fre$. Hence,
\begin{align}
  h(z) = \int_{\cup_{\ell\in I_1} K_\ell}g_\fre(x,x_*)d\om_{\fre_n}(x,z)
  &\le \sum_{\ell\in I_1}\sup_{x\in K_\ell}g_\fre(x,x_*)\om_{\fre_n}(K_\ell,z).
\end{align}
Since $g_\fre(z,x_*)$ is concave on every gap $K_\ell$ of $\fre$ that does not contain $x_*$, regularity of $\fre$ implies that the supremum of $g_\fre(z,x_*)$ on $K_\ell$ is attained at the critical point $\ze_\ell\in K_\ell$. Consequently,
\begin{align}
  h(z) \le  \sum_{\ell\in I_1} g_\fre(\ze_\ell,x_*)\om_{\fre_n}(K_\ell,z),
\end{align}
and applying this inequality at $z=\infty$ and $z=c_1,\dots,c_{r_n}$, we can use \eqref{om-est-res} to obtain
\begin{align}\label{PW-est1}
  & (d_n-r_n) h(\infty) + \sum_{j=1}^{r_n} h(c_j) \no\\
  &\quad \le
  \sum_{\ell\in I_1} g_\fre(\ze_\ell,\infty) \bigg[(d_n-r_n)\om_{\fre_n}(K_\ell,\infty) + \sum_{j=1}^{r_n}\om_{\fre_n}(K_\ell,c_j)\bigg] \no\\
  &\quad \le \sum_{\ell\in I_1} g_\fre(\ze_\ell,\infty).
\end{align}

Since $\{c_j\}_{j=r_n+1}^m$ are the zeros of $T_{n,w_0,x_*}$, they are contained in distinct gaps $\{K_\ell\}_{\ell\in I_2}$ of $\fre$. We thus have
\begin{align}\label{PW-est2}
  \sum_{j=r_n+1}^m g_\fre(c_j,x_*) \le \sum_{\ell\in I_2}^m g_{\fre}(\ze_\ell,x_*).
\end{align}
By Theorem~\ref{En-thm}, the gaps $\{K_\ell\}_{\ell\in I_2}$ do not intersect $\fre_n$ and are therefore distinct from the gaps $\{K_\ell\}_{\ell\in I_1}$.

If $\deg(T_{n,w_0,x_*})=n-1$, then, by Theorem~\ref{En-thm}, the unbounded gap $K_{\ell_0}$ of $\fre$ does not intersect $\fre_n$. Moreover, by the alternation theorem, $T_{n,w_0,x_*}$ has no zeros in $K_{\ell_0}$, and hence $K_{\ell_0}$ is distinct from the gaps $\{K_\ell\}_{\ell\in I\cup I_1}$. Let $I=I_1\cup I_2$ if $\deg(T_{n,w_0,x_*})=n$ and $I=I_1\cup I_2\cup\{\ell_0\}$ if $\deg(T_{n,w_0,x_*})=n-1$. Then combining the estimate $g_\fre(\infty,x_*)\le g_\fre(\ze_{\ell_0},x_*)$ with \eqref{PW-est1}, \eqref{PW-est2}, and \eqref{long-exp}, we get
\begin{align}
  &(n-m)g_\fre(\infty,x_*) - (d_n-r_n)g_{\fre_n}(\infty,x_*) + \sum_{j=1}^{m}g_{\fre}(c_j,\infty) - \sum_{j=1}^{r_n} g_{\fre_n}(c_j,\infty) \no \\
  &\quad \le
  \sum_{\ell\in I} g_\fre(\ze_\ell,\infty)
  \le PW(\fre,x_*).
\end{align}
Substituting this inequality into \eqref{Wn-gn-ineq} finally yields the upper bound \eqref{wUB-res}.
\end{proof}

\section{Asymptotic Bounds}\label{s3}

In this section we derive asymptotic bounds for the weighted Widom factors when $\fre\subset\bbR$ and $x_*\in\ol\bbR\bs\fre$.
We start by deriving an asymptotic lower bound which improves the universal lower bound \eqref{wSzLB-res}.

\begin{theorem}
Let $\fre\subset\bbR$ be a compact set of positive capacity, and let $x_*$ be a point in $\ol\bbR\bs\fre$. Assume that the weight $w: \fre\to [0,\infty)$ is continuous at a.e.\ point of $\fre$ and that $w\ge|P|$ on $\fre$ for some polynomial $P\not\equiv 0$. Then
\begin{align}\label{wALB-res}
  \liminf_{n\to\infty} W_n(\fre,w,x_*) \ge 2S(\fre,w,x_*).
\end{align}
\end{theorem}
\begin{remark}
If $w$ is continuous on $\fre$, then $w\ge|P|$ is equivalent to the conditions that $w$ has at most finitely many zeros on $\fre$ and at each such zero $x_0$,
\begin{align}
\limsup_{x\to x_0 \atop x\in\fre}\frac{\log w(x)}{\log|x-x_0|}<\infty.
\end{align}
Equivalently, there exists an $\al>0$ such that $w(x)\ge|x-x_0|^\al$ in a small neighborhood of $x_0$.
\end{remark}
\begin{proof}
Since both $W_n(\fre,w,x_*)$ and $S(\fre,w,x_*)$ scale linearly with $w$, we may assume without loss of generality that $P(x_*)=1$.

Let $f(x)=|P(x)|/w(x)$ for all $x\in\fre$ where $w(x)\neq0$, and set $f(x)=1$ when $w(x)=0$. Then $0\le f\le 1$ on $\fre$ and $f$ is continuous at a.e.\ point of $\fre$. It follows that there are polynomials $\{P_m\}_{m=1}^\infty$ with zeros outside of $\fre$ such that $f\le P_m\le2$ on $\fre$ and $P_m\to f$ pointwise a.e.\ on $\fre$.

Let $w_m=1/{P_m}$ and $\wti w_m=|P|w_m$. Then for each $m\ge1$, we have
\begin{align}
   \frac1f=\frac{w}{|P|}\ge\frac1{P_m}\ge\frac12,
\end{align}
and hence $w \ge \wti w_m \ge \frac12|P|$ on $\fre$ and $\wti w_m\to w$ a.e.\ on $\fre$. Let $d=\deg(P)$. Since $w \ge \wti w_m$ and $P(x_*)=1$, we have for all $n$ that
\begin{align}
  \|w T_{n,w,x_*}\|_\fre \ge \|\wti w_m T_{n,w,x_*}\|_\fre
  = \|w_m P T_{n,w,x_*}\|_\fre \ge \|w_m T_{n+d,w_m,x_*}\|_\fre.
\end{align}
Therefore, $W_n(\fre,w,x_*) \ge W_{n+d}(\fre,w_m,x_*)$.
		
By Theorem~\ref{wLB-thm}, $\liminf_{n\to\infty}W_n(\fre,w_m,x_*)\ge2S(\fre,w_m,x_*)$, and hence
\begin{align}\label{W-Swm-lb}
  \liminf_{n\to\infty}W_n(\fre,w,x_*) \ge \liminf_{n\to\infty}W_{n+d}(\fre,w_m,x_*) \ge 2S(\fre,w_m,x_*).
\end{align}
As in Lemma~\ref{Sw-pol-lem}, it is easy to see that $w_0=|P|$ is a Szeg\H{o} class weight. Since $\frac12 w_0\le w_m \le w$ and $w$ is bounded, it therefore follows from the dominated convergence theorem that $S(\fre,w_m,x_*)\to S(\fre,w,x_*)$ as $m\to\infty$. Thus, taking $m\to\infty$ in \eqref{W-Swm-lb} yields \eqref{wALB-res} and the proof is complete.
\end{proof}

Next, we show an asymptotic upper bound which holds for very general weights. In particular, this asymptotic upper bound holds for weights with arbitrary zeros and for weights that are not of Szeg\H{o} class.

\begin{theorem}
Let $\fre\subset\bbR$ be a regular and compact Parreau--Widom set, $x_*$ a point in $\ol\bbR\bs\fre$, and $w$ an upper semi-continuous weight on $\fre$. Then
\begin{align}\label{wAUB-res}
  \limsup_{n\to\infty} W_n(\fre,w,x_*) \le 2S(\fre,w,x_*)\exp[PW(\fre,x_*)].
\end{align}
\end{theorem}
\begin{proof}
Let $\eps>0$ be given. Since $w$ is upper semi-continuous, there exists some constant $C>0$ such that $w\le C$ on $\fre$. Hence the function $\frac{1}{w(x)+\eps}$ is lower semi-continuous on $\fre$ and it satisfies that
\begin{align}
   \frac{1}{C+\eps}\le \frac{1}{w(x)+\eps}\le \frac{1}{\eps}.
\end{align}
It follows that there exist polynomials $\{P_m\}_{m=1}^\infty$ such that
\begin{align}
   \frac1{2(C+\eps)}\le P_m\le \frac{1}{w+\eps} \; \mbox{ on $\fre$}
\end{align}
and $|P_m-\frac{1}{w+\eps}|\to0$ pointwise on $\fre$.
Let $w_m=1/{P_m}$ and note that for each $m\ge1$, Theorem~\ref{wUB-thm} implies that
\begin{align}
  \limsup_{n\to\infty}W_n(\fre,w_m,x_*) \le 2S(\fre,w_m,x_*)\exp[PW(\fre)].
\end{align}
Since, by construction, $w\le w+\eps\le w_m$ on $\fre$, we have $W_n(\fre,w,x_*)\le W_n(\fre,w_m,x_*)$ for all $n$, and hence
\begin{align}\label{W-Swm-ub}
  \limsup_{n\to\infty}W_n(\fre,w,x_*) \le 2S(\fre,w_m,x_*)\exp[PW(\fre)].
\end{align}
By construction, we also have $\eps\le w_m\le 2(C+\eps)$ on $\fre$ and $w_m\to w+\eps$ pointwise on $\fre$. Hence, $S(\fre,w_m,x_*)\to S(\fre,w+\eps,x_*)$ as $m\to\infty$ by the dominated convergence theorem and $S(\fre,w+\eps,x_*)\to S(\fre,w,x_*)$ as $\eps\to0$ by the monotone convergence theorem. Thus, taking $m\to\infty$ and then $\eps\to0$ in \eqref{W-Swm-ub} yields \eqref{wAUB-res}.
\end{proof}

As a corollary of the asymptotic upper bound \eqref{wAUB-res} and the universal lower bound \eqref{wSzLB-res}, we get the following Szeg\H{o}-type theorem.

\begin{theorem}
Let $\fre\subset\bbR$ be a regular and compact Parreau--Widom set, $x_*$ a point in $\ol\bbR\bs\fre$, and $w$ an upper semi-continuous weight on $\fre$. Then
\begin{align}
  (a) \; \inf_{n}W_n(\fre,w,x_*) > 0 \quad\Longleftrightarrow\quad
  (b) \; \int\log w(z)\,d\rho_\fre(z)>-\infty.
\end{align}
Moreover, if either $(a)$ or $(b)$ holds, then also $\sup_{n}W_n(\fre,w,x_*)<\infty$.
\end{theorem}


\end{document}